\documentclass{article}

\usepackage{amsmath, amssymb, mathtools, amsfonts}

\usepackage{times}
\usepackage{bm}
\usepackage{natbib}
\usepackage{hyperref}

\usepackage{relsize}

\usepackage{amsthm, thmtools, thm-restate}

\usepackage[linesnumbered,ruled,vlined]{algorithm2e}

\newtheorem{theorem}{Theorem}
\newtheorem{lemma}{Lemma}
\newtheorem{corollary}{Corollary}

\newtheorem{definition}{Definition}
\newtheorem*{remark}{Remark}
\newtheorem*{theoremintro}{Theorem}

\newtheorem{theoremsupp}{Theorem}
\newtheorem{lemmasupp}{Lemma}
\newtheorem{corollarysupp}{Corollary}

\newcommand{\mybrk}{ \\ }

\makeatletter
\newcommand{\p}{\@ifstar{\@pstar}{\@pnostar}}
\newcommand{\@pstar}[1]{\smash{({#1})}}
\newcommand{\@pnostar}[1]{
    \relax\if@display
        \expandafter{\left({#1}\right)}
    \else
        \expandafter{({#1})}
    \fi
}
\makeatother

\makeatletter
\newcommand{\brk}{\@ifstar{\@brkstar}{\@brknostar}}
\newcommand{\@brkstar}[1]{\smash{[{#1}]}}
\newcommand{\@brknostar}[1]{
    \relax\if@display
        \expandafter{\left[{#1}\right]}
    \else
        \expandafter{[{#1}]}
    \fi
}
\makeatother

\makeatletter
\newcommand{\set}{\relax\@ifstar{\@setstar}{\@setnostar}}
\newcommand{\@setstar}[1]{\relax\smash{\{{#1}\}}}
\newcommand{\@setnostar}[1]{
    \relax\if@display
        \expandafter{\left\{{#1}\right\}}
    \else
        \expandafter{\{{#1}\}}
    \fi
}
\makeatother

\makeatletter
\newcommand{\abs}[1]{
    \relax\if@display
        \expandafter{\left|{#1}\right|}
    \else
        \expandafter{|{#1}|}
    \fi
}
\makeatother

\makeatletter
\newcommand{\norm}[2][]{
    \relax\if@display
        \left\|{#2}\right\|_{#1}
    \else
        \|{#2}\|_{#1}
    \fi
}
\makeatother

\newcommand{\nats}{\mathbb N}
\newcommand{\reals}{\mathbb R}

\newcommand{\iid}{\overset{\mathrm{iid}}{\sim}}

\DeclareMathOperator{\rank}{rank}
\DeclareMathOperator{\nullity}{Nullity}
\DeclareMathOperator{\tr}{tr}

\newcommand{\ELPS}{{\mskip -1mu ... \mskip 2mu}}

\newcommand{\SYM}{\mathcal{S}}
\newcommand{\PP}{\mathbb{P}^{\mathsmaller{+}}}
\newcommand{\PPP}{\mathbb{P}^{\mathsmaller{+}{\mskip -3mu}\mathsmaller{+}}}
\newcommand{\DZ}{\mathcal{D}^{\mathsmaller{0}}}
\newcommand{\DP}{\mathcal{D}^{\mathsmaller{+}}}
\newcommand{\DPP}{\mathcal{D}^{\mathsmaller{+}{\mskip -3mu}\mathsmaller{+}}}

\newcommand{\SYMG}[1][]{\mathcal{S}_{\mathcal{G}_{#1}}}
\newcommand{\PGP}[1][]{\mathbb{P}_{{\mskip -1.5mu}\mathcal{G}_{#1}}^{\mathsmaller{+}}}
\newcommand{\PGPP}[1][]{\mathbb{P}_{{\mskip -1.5mu}\mathcal{G}_{#1}}^{\mathsmaller{+}{\mskip -3mu}\mathsmaller{+}}}
\newcommand{\DGZ}[1][]{\mathcal{D}_{{\mskip -.8mu}\mathcal{G}_{#1}}^{{\mskip 1.2mu}\mathsmaller{0}}}
\newcommand{\DGP}[1][]{\mathcal{D}_{{\mskip -.8mu}\mathcal{G}_{#1}}^{\mathsmaller{+}}}
\newcommand{\DGPP}[1][]{\mathcal{D}_{{\mskip -.8mu}\mathcal{G}_{#1}}^{\mathsmaller{+}{\mskip -3mu}\mathsmaller{+}}}
\newcommand{\RGP}[1][]{\mathcal{R}_{{\mskip -.8mu}\mathcal{G}_{#1}}^{\mathsmaller{+}}}
\newcommand{\RGPP}[1][]{\mathcal{R}_{{\mskip -.8mu}\mathcal{G}_{#1}}^{\mathsmaller{+}{\mskip -3mu}\mathsmaller{+}}}

\newcommand{\KSYMG}[2][]{K_{#2}{\mskip 1mu}{\cap}{\mskip 2.7mu}\SYMG[#1]}
\newcommand{\KPGP}[2][]{K_{#2}{\mskip 1mu}{\cap}{\mskip 2.7mu}\PGP[#1]}
\newcommand{\KDGZ}[2][]{K_{#2}{\mskip 1mu}{\cap}{\mskip 1.8mu}\DGZ[#1]}
\newcommand{\KDGP}[2][]{K_{#2}{\mskip 1mu}{\cap}{\mskip 1.8mu}\DGP[#1]}
\newcommand{\KRGP}[2][]{K_{#2}{\mskip 1mu}{\cap}{\mskip 1.8mu}\RGP[#1]}

\newcommand{\SUCCEQ}{{\mskip 2mu}{\succeq}{\mskip 2mu}}
\newcommand{\TIMES}{{\mskip 1mu}{\times}{\mskip 1mu}}
\newcommand{\IN}{{\mskip 3mu}{\in}{\mskip 3mu}}
\newcommand{\LOG}{\log{\mskip -2mu}}
\newcommand{\RANK}{\rank{\mskip -2mu}}
\newcommand{\TR}{\tr{\mskip -2mu}}
\newcommand{\DIM}{\dim{\mskip -2mu}}

\newcommand{\MIN}{\min{\mskip -2mu}}

\DeclareMathOperator{\con}{C}
\DeclareMathOperator{\gaus}{G}
\DeclareMathOperator{\spl}{H}

\renewcommand{\cite}{\citep}

\title{Pseudo-likelihood Estimators for Graphical Models: Existence and Uniqueness}
\author{Benjamin Roycraft and Bala Rajaratnam}
\date{\today}

\begin{document}

\maketitle

\begin{abstract}
    Graphical and sparse (inverse) covariance models have found widespread use in modern sample-starved high dimensional applications. A part of their wide appeal stems from the significantly low sample sizes required for the existence of estimators, especially in comparison with the classical full covariance model. For undirected Gaussian graphical models, the minimum sample size required for the existence of maximum likelihood estimators had been an open question for almost half a century, and has been recently settled \citep{BenDavid2015, uhler_geometry_2012, gross_maximum_2018}. The very same question for pseudo-likelihood estimators has remained unsolved ever since their introduction in the '70s. Pseudo-likelihood estimators have recently received renewed attention as they impose fewer restrictive assumptions and have better computational tractability, improved statistical performance, and appropriateness in modern high dimensional applications, thus renewing interest in this longstanding problem. In this paper, we undertake a comprehensive study of this open problem within the context of the two classes of pseudo-likelihood methods proposed in the literature. We provide a precise answer to this question for both pseudo-likelihood approaches and relate the corresponding solutions to their Gaussian counterpart.
\end{abstract}

\section{Introduction}

Graphical and sparse (inverse) covariance models have become a staple in modern statistics and machine learning \citep{lauritzen_graphical_1996}. As statistical models, they enjoy many properties which are attractive in the context of modern high-dimensional sample-starved applications. One such property is the low sample size requirements for the existence of maximum likelihood estimators, facilitating estimation when the features vastly outnumber the available samples. As a concrete example, consider the classical inverse covariance estimation problem for samples $X_1,\ELPS, X_n \iid \mathcal{N} \p{0,\Sigma_{p \TIMES p}}$. It is well known that for the \textsc{mle} of $\Sigma^{-1}$ to exist, a.s. $n\geq p$. Thus, in modern high-dimensional sample-starved settings where the number of covariates $p$ is much larger than the sample size $n$, full covariance models are often not very useful. On the contrary, if the sparsity pattern in the inverse covariance matrix corresponds to a chordal/decomposable graph, the sample size required is only the size of the maximal clique \citep{lauritzen_graphical_1996}, rendering these models highly applicable in modern high-dimensional regimes.

Ever since their introduction \citep{dempsterCovarianceSelection1972}, the precise sample size which guarantees the existence of the \textsc{mle} for undirected Gaussian graphical models has attracted much interest from the statistics community \citep{lauritzen_mixed_1989, 10.2307/4616281, BenDavid2015, bernstein_maximum_2022, bernstein_computing_2022, gross_maximum_2018, blekherman_maximum_2019, uhler_geometry_2012} and had been an open problem for almost half a century. A comprehensive treatment for undirected Gaussian graphical models was recently provided by \citep{BenDavid2015}, who successfully demonstrated that the minimum sample size required is bounded above by the graph degeneracy + 1 and below by the subgraph connectivity + 1. Likewise \citep{gross_maximum_2018, blekherman_maximum_2019} establish an improved upper bound based on the generic completion rank. The very same question on the existence of pseudo-likelihood estimators has remained unsolved since their introduction in the '70s \citep{besagStatisticalAnalysisNonLattice1975}. Pseudo-likelihood estimators have recently received renewed attention as they impose fewer restrictive assumptions and have better computational tractability, improved statistical performance, and appropriateness in modern high dimensional applications, especially in comparison with their Gaussian counterpart (see \cite{10.2307/24775310} and the references therein). Recently, two major classes of pseudo-likelihood methods have been proposed (see \cite{10.2307/24775310, lee_learning_2015}) and have been shown to perform well in the modern high-dimensional setting. These have been popularly referred to as the ``\textsc{concord}'' and ``\textsc{space}'' estimators (for reasons that will become evident, we shall work with the convex version of the latter, herein referred to as ``\textsc{conspace}''). Despite the efficacy of these estimators, very little is known about the precise conditions required for their existence - which is the primary motivation of our work. In particular, we consider the two classes of pseudo-likelihood estimators above and comprehensively analyze the sample size regimes in which they are guaranteed to exist.

We now summarize the main results of the paper. For a given graph $\mathcal{G}$, let $\delta$ be the graph degeneracy, $\kappa^*$ the subgraph connectivity, and $\ell$ the generic completion rank. Let $n_{\textsc{concord}}$ and $n_{\textsc{conspace}}$ denote the minimum sample sizes, equivalently the rank of sample covariance matrix, required for existence and uniqueness of the \textsc{concord} and \textsc{conspace} pseudo-likelihood graphical model estimators. Our first result precisely quantifies the minimum sample size requirements:
\begin{theoremintro}
 $\kappa^*+1 \leq n_{\textsc{concord}} = n_{\textsc{conspace}}\leq \ell \leq \delta+1$.
\end{theoremintro}

Although the functional forms of the \textsc{concord} and \textsc{conspace} pseudo-likelihoods are different, surprisingly both estimators are identical in terms of sample size requirements. Moreover, and despite the different formulations of the pseudo-likelihood and Gaussian likelihood estimation problems, the above bounds match the best known bounds for the Gaussian case \citep{gross_maximum_2018, BenDavid2015, uhler_geometry_2012}. This equivalence raises a natural question of comparison between the necessary sample sizes for the pseudo-likelihood ($n_{\text{pseudo}}\coloneqq n_{\textsc{concord}}=n_{\textsc{conspace}}$) and Gaussian likelihood ($n_{\mathrm{Gaussian}}$) methods, which is answered by our second result:
\begin{theoremintro}
 $n_{\mathrm{Gaussian}} \leq n_{\mathrm{pseudo}}$.
\end{theoremintro}

This result indicates that the minimum sample size required in the Gaussian setting is the same or fewer than for the pseudo-likelihood approach, giving the impression that the latter is inferior. However, as will be seen in this work, the sample size requirements are exactly the same for commonly studied classes of graphs covered in the literature. The two preceding results are conveniently summarized in the following:
\begin{theoremintro}
 $\kappa^*+1 \leq n_{\mathrm{Gaussian}} \leq n_{\mathrm{pseudo}}\leq \ell \leq \delta+1.$
\end{theoremintro}

As well as solving the sample size problem, we show that, like their Gaussian counterparts, pseudo-likelihood estimators exhibit an important monotonicity property: a sub-graph of an undirected graphical model requires fewer samples - rigorously demonstrating the beneficial effect of sparsity on pseudo-likelihood based graphical model estimation. 

\defcitealias{RoycraftSupp2023}{Supplement}

This note is organized as follows: Section~\ref{section::prelims} introduces notation and reviews relevant previous work. Section~\ref{section::general} provides general theorems with regards to the existence of both (Gaussian) likelihood and pseudo-likelihood estimators. Section~\ref{section::ranks} specifies the main results of the paper regarding the minimum sample size required for existence of pseudo-likelihood estimators. Proofs of all main results may be found in Appendix~\ref{section::proofs}. Appendix~\ref{section::exact_compute} discusses how the theoretical criterion of Section~\ref{section::general} can be leveraged computationally to further improve the upper bounds presented in Section~\ref{section::ranks}.

\section{Preliminaries, Notation, and Literature Review} \label{section::prelims}

Let a given sample of size $n$ of $p$ variables be represented by the data matrix $\mathbf{X} = \brk{x_{1\cdot }, \ELPS, x_{n\cdot}}^{\mathsmaller{\top}} = \brk{x_{\cdot 1}, \ELPS, x_{\cdot p}}$. Assume each observation has mean $0$ and covariance matrix $\Sigma \succ 0$. In this work we consider the problem of estimating the inverse covariance matrix $\Sigma^{-1}$ under graphical (i.e. sparsity) constraints, focusing on two classes of pseudo-likelihood approaches and the corresponding sample size required for each.

We now provide is a comprehensive list of the common notation and concepts we will use throughout the paper. Let $S \coloneqq \frac1n\mathbf{X}^{\mathsmaller{\top}}\mathbf{X}$ be the sample covariance matrix under a mean-zero assumption. $\mathcal{G} \equiv \p{V, E}$ denotes a graph with vertices $V = V\p{\mathcal{G}} = \set{1, \ELPS, p}$ and edges $E = E\p{\mathcal{G}} \subseteq \set{\set{i, j}\colon i,j \IN V, \ i \neq j}$. We say that a symmetric matrix $\Omega$ satisfies the graphical/sparsity constraints associated with $\mathcal{G}$ if $\set{i, j} \notin E \implies \Omega_{ij} = \Omega_{ji} = 0$. Let $\Omega_{ij}$ denote the individual elements of $\Omega$ and $\Omega_{\cdot 1}, \ELPS, \Omega_{\cdot p}$ the column vectors. Also, $\Omega_{\mathsmaller{D}}$ denotes the diagonal matrix such that $\p{\Omega_D}_{ii}=\Omega_{ii}$.

A symmetric matrix $A$ is said to be in \textit{general position} if every $q \TIMES q$ principal submatrix of $A$ has rank $\MIN\set{q, \RANK\p{A}}$. Under general conditions, including when $\mathbf{X}$ is absolutely continuous in the multivariate sense, the $S = \frac{1}{n}\mathbf{X}^{\mathsmaller{\top}}\mathbf{X}$ is in general position almost surely.

Denote by $\mathcal{M}$ the space of real $p \TIMES p$ matrices; likewise $\SYM\coloneqq\set{\Omega\IN \mathcal{M}\colon \Omega = \Omega^{\mathsmaller{\top}}}$ is the corresponding space of symmetric matrices. $\PP \coloneqq \set{\Omega \IN \SYM\colon \Omega \SUCCEQ 0}$ and $\PPP \coloneqq \set{\Omega \IN \SYM\colon \Omega \succ 0}$ are the positive semidefinite and definite matrices, respectively. Similarly $\DP \coloneqq \set{\Omega \IN \SYM\colon \Omega_{ii} \geq 0 \ \forall i=1,\ELPS, p}$, $\DPP \coloneqq \set{\Omega \IN \SYM\colon \Omega_{ii} > 0 \ \forall i=1,\ELPS, p}$, and $\DZ \coloneqq \set{\Omega \IN \SYM\colon \Omega_{ii} = 0 \ \forall i=1,\ELPS, p}$. For a given graph $\mathcal{G}$, $\SYMG \coloneqq \mybrk \set{\Omega \IN \SYM\colon \Omega_{ij} = \Omega_{ji} = 0 \ \forall \set{i,j} \notin E\p{\mathcal{G}}}$ is the subset of symmetric matrices which satisfy the graphical constraints induced by $\mathcal{G}$. Let $\DGZ \coloneqq \DZ \cap \SYMG$. Note that $\SYM, \SYMG$, and $\DGZ$ constitute linear subspaces within $\mathcal{M}$. Finally, define the cones $\PGP \coloneqq \SYMG \cap \PP$, $\PGPP \coloneqq \SYMG \cap \PPP$, $\DGP \coloneqq \SYMG \cap \DP$, and $\DGPP \coloneqq \SYMG \cap \DPP$; each is closed and convex.

Two major classes of pseudo-likelihood approaches have been proposed, and are popularly referred to as ``\textsc{space}'' and ``\textsc{concord}''. The first method, \textsc{space} \citep{Peng2009}, is equivalent \citep[Lemma~1]{10.2307/24775310} to minimization of the following pseudo-likelihood:
\begin{equation} \label{spaceobjective}
 -\frac{n}{2}\sum_{i=1}^p\LOG\p{\Omega_{ii}} + \sum_{i=1}^p\frac{w_i}{\Omega_{ii}^2}\Bigg\lVert\Omega_{ii}x_{\cdot i} - \sum_{j\neq i}\Omega_{ij} x_{\cdot j}\Bigg\rVert^2 
\end{equation}
where the weights $w_i$ can either take the value $w_i = \Omega_{ii}$ (``partial variance weights'') or $w_i = 1$ (``uniform weights''). Three other methods proposed in the literature, namely \textsc{symlasso} \citep{Friedman2010ApplicationsOT}, \textsc{splice} \citep{Rocha2008}, and the method of \cite{lee_learning_2015} are, up to reparametrization, equivalent to \textsc{space} with partial variance weights (see \cite{10.2307/24775310}). Thus, these four seemingly different methods can be treated in a unified manner. Since the parameterization of \cite{lee_learning_2015} is the only one of the four to yield a convex objective, we shall focus our analyses on this version, referring to it as ``\textsc{conspace}'' (to emphasize that it corresponds to a convex reformulation of the \textsc{space} method).

As mentioned above, \cite{Peng2009} also proposes to use uniform weights, i.e. $w_i = 1$. However, a cursory analysis shows that \eqref{spaceobjective} is never bounded below in this case, thus no minimizer exists: Along any ray in $\DGP$ (including in the direction of the identity), the first term is negative and unbounded, whereas the second term remains constant. Thus, uniform weights do not lead to a meaningful estimate, and as such we do not consider this case in our subsequent analyses.

The second major method, \textsc{concord} \citep{10.2307/24775310}, corresponds to the case of weights $w_i = \Omega_{ii}^2$, yielding a convex pseudo-likelihood objective. The \textsc{concord} objective has a different functional form than that of \textsc{conspace}, even under reparametrization, and has been shown to yield stable and reliable graphical model estimates.

In what follows, we consider the problem of minimizing the objectives corresponding to \textsc{concord} and \textsc{conspace}, subject to graphical constraints. Recall the specific \textsc{concord} objective:
\begin{align}
 & -\frac{n}{2}\sum_{i=1}^p\LOG\p{\Omega_{ii}^2} + \sum_{i=1}^p\Omega_{ii}^2\Bigg\lVert x_{\cdot i} + \sum_{j\neq i}\frac{\Omega_{ij}}{\Omega_{ii}} x_{\cdot j}\Bigg\rVert^2 \nonumber \\
 = \ & n\p{-\sum_{i=1}^p\LOG\p{\Omega_{ii}^2} + \sum_{i=1}^p \Omega_{\cdot i}^{\mathsmaller{\top}}S\Omega_{\cdot i}} \label{eq::1}.
\end{align}
The \textsc{concord} estimate following the graphical constraints of $\mathcal{G}$ is then the minimizer of \eqref{eq::1} on $\DGPP$. This objective is convex on $\DGPP$, but may not be strictly convex. As benefits later analyses, the form of \eqref{eq::1} allows explicit dependence of the objective on the sample $\mathbf{X}$ to be suppressed. In particular, define the generalized \textsc{concord} objective as follows, dependent on $A \SUCCEQ 0$:
\begin{align}
 \con_{A}\p{\Omega}\coloneqq{} & -2\sum_{i=1}^p\LOG\p{\Omega_{ii}} + \sum_{i=1}^p \Omega_{\cdot i}^{\mathsmaller{\top}}A\Omega_{\cdot i} \nonumber \\
 ={} & -2\LOG\p{\abs{\Omega_D}} + \tr\p{\Omega^{\mathsmaller{\top}}A\Omega}. \nonumber
\end{align}

This generalized objective is convex for any $A \SUCCEQ 0$. The \textsc{concord} estimator is alternatively the minimizer of $\con_{S}$ on $\DGPP$. Similarly, the generalized \textsc{conspace} objective is
\begin{align}
 \spl_{A}\p{\Omega}\coloneqq{} & -2\sum_{i=1}^p\LOG\p{\Omega_{ii}} + \sum_{i=1}^p \frac{1}{\Omega_{ii}}\Omega_{\cdot i}^{\mathsmaller{\top}}A\Omega_{\cdot i} \nonumber \\
 ={} & -2\LOG\p{\abs{\Omega_D}} + \tr\p{\Omega_D^{-1}\Omega^{\mathsmaller{\top}}A\Omega}. \nonumber
\end{align}

This objective is also convex, but not necessarily strictly convex, and minimizing $\spl_{S}$ over $\DGPP$ gives the graph-constrained \textsc{conspace} estimator. By divorcing the two objectives above from explicit dependence on the sample covariance matrix, we may analyze the optimization problems holistically, depending only on summary properties of $A$. This abstraction is important, as our results can be stated without invoking any particular sampling distribution on $\mathbf{X}$.

Later we compare our results to Gaussian maximum likelihood estimation. For $A \SUCCEQ 0$ define
\begin{equation}
	\gaus_{A}\p{\Omega}\coloneqq- \LOG\abs{\Omega}+ \tr\p{A\Omega}. \nonumber
\end{equation}

Subject to the graphical constraints implied by $\mathcal{G}$, the Gaussian \textsc{mle} is the minimizer on $\PGPP$ of $\frac{n}{2}\p{p\LOG\p{2\pi} - \LOG\abs{\Omega} + \sum_{i=1}^{n} x_{\cdot i}^{\mathsmaller{\top}}\Omega x_{\cdot i}}$, which equivalently minimizes $\gaus_{S}$.

\section{General Characterization of Solutions} \label{section::general}

First we establish general abstract conditions under which the pseudo-likelihood and Gaussian \textsc{mle} exist and are unique. For $A \SUCCEQ 0$, define $K_{A} \coloneqq \set{\Omega \IN \mathcal{M}\colon A \Omega = 0}$.

\begin{lemma} \label{lemma::general}
	Given $A \SUCCEQ 0$, consider either pseudo-likelihood objective $\con_{A}$ or $\spl_{A}$. Then a corresponding minimizer exists on $\DGPP$ if and only if
	\begin{equation}
		\KDGP{A} = \KDGZ{A}. \nonumber
	\end{equation}
	
	Furthermore, if a minimizer $\Omega^* \IN \DGPP$ exists, the set of minimizers is affine of the form
	\begin{equation}
		\set{\Omega^*+\Phi\colon \Phi \IN \KDGZ{A}}. \nonumber
	\end{equation}
\end{lemma}

Lemma~\ref{lemma::general} gives a geometric condition for the existence and uniqueness of the \textsc{concord} and \textsc{conspace} estimators. Both exist and are unique if and only if $\KDGP{S} = \KDGZ{S} = \set{0}$, where $S$ the sample covariance matrix . A similar result holds for the Gaussian likelihood:

\begin{lemma} \label{lemma::general_2}
	Given $A \SUCCEQ 0$, a unique minimizer of $G_{A}$ exists on $\PGPP$ if and only if $$\KPGP{A}{\mskip 1mu}{=}{\mskip 1mu}\set{0}.$$
\end{lemma}

Due to the log-determinant barrier term, the Gaussian likelihood is always strictly convex, and any minimizer (if it exists) is guaranteed to be unique; thus Lemma~\ref{lemma::general_2} primarily characterizes existence. Namely, the (unique) Gaussian \textsc{mle} will exist if and only if $\KPGP{S} = \set{0}$.

Intuitively, the criteria of Lemmas~\ref{lemma::general} and \ref{lemma::general_2} encode how $\con_{A}$, $\spl_{A}$, and $\gaus_{A}$ behave as $\norm[\mathrm{F}]{\Omega}\to \infty$. Due to barrier terms, each objective approaches $\infty$ near the boundary of $\DGPP$ or $\PGPP$. As such, from continuity a minimizer will exist over any bounded region. The mode by which a minimizer on an \textit{unbounded} set may fail to exist is then dictated by the behavior as $\norm[\mathrm{F}]{\Omega}\to\infty$; here the log-barrier terms are unbounded below, whereas $\tr\p{\Omega^TA\Omega}$, $\tr\p{\Omega_D^{-1}\Omega^TA\Omega}$, and $\tr\p{A\Omega}$ are unbounded above. Interplay between these two competing elements determines if a minimizer exists or not. In particular, when the trace term uniformly dominates as $\norm[\mathrm{F}]{\Omega}\to\infty$, a minimizer is guaranteed to exist. The sets $K_{A}$ and $\DGZ$ influence these dynamics: For the \textsc{concord} and Gaussian cases, $K_{A}$ and $\DGZ$ include, respectively, all the directions along which the trace terms and barrier $\LOG\p{\abs{\Omega_{D}}}$ are constant. Only if the trace term is constant and the log term is not (as for directions in $\KDGP{A}\setminus \KDGZ{A}$) will the latter dominate, in which case no minimizer will exist. $\KDGZ{A}$ also includes the directions along which the objectives overall are constant, influencing uniqueness. Thus, $K_{A}$ and $\DGZ$ jointly determine both existence and uniqueness of the minimizer. The \textsc{conspace} case behaves similarly, with additional technicalities.

\begin{remark}
	Lemmas~\ref{lemma::general} and \ref{lemma::general_2} together show a certain correspondence between the existence properties of the \textsc{concord}, \textsc{conspace}, and Gaussian maximum likelihood estimators. Disregarding log terms in the corresponding objectives, $\gaus_{A}$ is linear, $\con_{A}$ is quadratic, and $\spl_{A}$ is neither; as such, a correspondence is not to be expected in general. Although the objectives differ in form, the ``failure modes'' of each are however similar. The existence of a unique minimizer in all cases reduces to a convenient geometric arrangement relating the three methods. We shall leverage this novel discovery to characterize the minimum sample size required for each.
\end{remark}

\section{Convex pseudo-likelihood ranks: concord and conspace} \label{section::ranks}

In this section we quantify the sample size required for existence of unique pseudo-likelihood estimators in terms of the underlying graph $\mathcal{G}$. As the minimum sample size required is intimately linked to the rank of the sample covariance matrix, these requirements are first stated in terms of ranks, essentially serving as a proxy for the sample size. Towards this goal, define the following family of pseudo-likelihood ranks:

\begin{definition}[Weak \textsc{concord} rank]
 $\rho_{{\mskip 2mu}\textsc{concord}}\p{\mathcal{G}}$ is the smallest $\rho \IN \nats$ such that a unique minimizer of $\con_{A}$ exists on $\DGPP$ for almost every $A \SUCCEQ 0$ with $\RANK\p{A}\geq\rho$.
\end{definition}

\begin{definition}[Strong \textsc{concord} rank]
 $\rho_{{\mskip 2mu}\textsc{concord}}^{*}\p{\mathcal{G}}$ is the smallest $\rho^{*} \IN \nats$ such that a unique minimizer of $\con_{A}$ exists on $\DGPP$ for every $A \SUCCEQ 0$ in general position with $\RANK\p{A}\geq \rho^*$.
\end{definition}

\begin{definition}[Weak \textsc{conspace} rank]
 $\rho_{{\mskip 2mu}\textsc{conspace}}\p{\mathcal{G}}$ is the smallest $\rho \IN \nats$ such that a unique minimizer of $\spl_{A}$ exists on $\DGPP$ for almost every $A \SUCCEQ 0$ with $\RANK\p{A}\geq \rho$.
\end{definition}

\begin{definition}[Strong \textsc{conspace} rank]
 $\rho_{{\mskip 2mu}\textsc{conspace}}^{*}\p{\mathcal{G}}$ is the smallest $\rho^{*} \IN \nats$ such that a unique minimizer of $\spl_{A}$ exists on $\DGPP$ for every $A \SUCCEQ 0$ in general position with $\RANK\p{A} \geq \rho^*$.
\end{definition}

As $\RANK\p{A}$ increases, the minimization problems will transition between three regimes: First, a minimizer will not exist for any $A$, then exist only for some $A$ (and not for others), and eventually exist for almost every $A$. In particular, there is a point at which existence and uniqueness is guaranteed, a function of the underlying graph $G$. The pseudo-likelihood ranks as defined precisely quantify this behavior. The following lemma relates the four ranks above:

\begin{lemma} \label{thm::con_conspace}
	$\rho_{{\mskip 2mu}\textsc{concord}} = \rho_{{\mskip 2mu}\textsc{conspace}} \leq \rho_{{\mskip 2mu}\textsc{conspace}}^{*} = \rho_{{\mskip 2mu}\textsc{concord}}^{*}$.
\end{lemma}

Lemma~\ref{thm::con_conspace} follows immediately from the (surprising) results of Lemma~\ref{lemma::general}, showing that the existence properties of the \textsc{conspace} and \textsc{concord} estimators are completely equivalent. Thus we shall denote both of the \textsc{concord} and \textsc{conspace} \textit{pseudo-likelihood ranks} jointly as $\rho$ and $\rho^{*}$. Of the two definitions, the weak rank $\rho$ naturally characterizes existence and uniqueness of the \textsc{concord} and \textsc{conspace} estimators within a general statistical framework. More specifically, given very general conditions on the underlying sample $\mathbf{X}$, the induced distribution of the sample covariance matrix $S$ is absolutely continuous on $\SYM$ and $\RANK\p{S}\overset{\mathsmaller{\mathrm{a.s.}}} = n$. Thus, if $n \geq \rho$, existence and uniqueness of the estimators is guaranteed almost surely, and $\rho$ is the smallest value for which such a guarantee is possible. The following result formalizes how sample size requirements are related to the complexity/size of a graph:

\begin{corollary} \label{corr::subset}
    For any graphs $\mathcal{G}_1 \subset \mathcal{G}_2$, $\rho\p{\mathcal{G}_1} \leq \rho\p{\mathcal{G}_2}$ and $\rho^*\p{\mathcal{G}_1} \leq \rho^*\p{\mathcal{G}_2}$.
\end{corollary}

Both the weak and strong ranks increase with subset inclusion. Thus, if either can be calculated exactly, immediate bounds corresponding to related graphs follow based on the appropriate inclusion relations. This monotonicity also formalizes the intuition that estimation becomes less demanding as graphs become sparse, but not arbitrarily. Fewer edges does not immediately guarantee a smaller sample size requirement; the edge set has to be a proper subset.

In what follows, we give bounds for the weak and strong ranks in terms of graph degeneracy, subgraph connectivity number, and generic completion rank \citep{blekherman_maximum_2019}. First, define these graph characteristics formally: 

\begin{definition}[Graph Degeneracy]
 $\delta\p{\mathcal{G}}$ is the smallest value $\delta\IN \nats$ such that every subgraph of $\mathcal{G}$ contains a vertex of degree at least $\delta$.
\end{definition}

\begin{definition}[Disconnection Number]
 $\kappa\p{\mathcal{G}}$ is the smallest value $\kappa\IN \nats$ such that there exists a disconnected or single-element subgraph of $\mathcal{G}$ with $\abs{V\p{\mathcal{G}}}-\kappa$ vertices.
\end{definition}

\begin{definition}[Subgraph Connectivity Number]
 $\kappa^*\p{\mathcal{G}}$ is the smallest value $\kappa^*\IN \nats$ such that $\kappa\p{\mathcal{G}'}\leq \kappa^*$ for any subgraph $\mathcal{G}'$ of $\mathcal{G}$.
\end{definition}

\begin{definition}[Generic Completion Rank]
    $\ell\p{\mathcal{G}}$ is the smallest $\ell \IN \nats$ such that the projection of $\set{\Omega\in\mathcal{S}\colon \rank\p{\Omega}=\ell}$ onto $\SYMG$ is dense.
\end{definition}

See Appendix~\ref{section::algorithm} for an algorithmic description of the generic completion rank. To facilitate comparison with the Gaussian \textsc{mle}, we also define the weak and strong Gaussian ranks \citep{BenDavid2015}:

\begin{definition}[Weak Gaussian Rank]
 $\gamma\p{\mathcal{G}}$ is the smallest value $\gamma\IN \nats$ such that a unique minimizer of $\gaus_{A}$ exists on $\PGPP$ for almost every $A \SUCCEQ 0$ with $\RANK\p{A}\geq \gamma$.
\end{definition}

\begin{definition}[Strong Gaussian Rank]
 $\gamma^*\p{\mathcal{G}}$ is the smallest value $\gamma^*\IN \nats$ such that a unique minimizer of $\gaus_{A}$ exists on $\PGPP$ for every $A \SUCCEQ 0$ in general position with $\RANK\p{A}\geq \gamma^*$.
\end{definition}

A comprehensive study of the Gaussian ranks is found in \cite{BenDavid2015}. The weak Gaussian rank is equivalent to the notion of ``maximum likelihood threshold'' \citep{gross_maximum_2018}. In the same spirit, we come to our main theorems bounding pseudo-likelihood ranks.

\begin{theorem}[Upper bound] \label{theorem::boundsupper}
    For any graph $\mathcal{G}$, both $\rho \leq \ell \leq \delta+1$ and $\rho^* \leq \delta+1$.
\end{theorem}

\begin{theorem}[Lower bound] \label{theorem::boundslower}
    For any graph $\mathcal{G}$, both $\gamma \leq \rho $ and $\gamma^* \leq \rho^* $.
\end{theorem}

Theorem~\ref{theorem::boundslower} asserts that both pseudo-likelihood ranks are bounded below by the corresponding Gaussian equivalents. Thus, in general the sample size required for \textsc{concord} and \textsc{conspace} is at least that of the Gaussian \textsc{mle}, giving the impression that it may be higher. However, as will be shown, the sample size requirements are essentially the same for most practical purposes, and in fact are exactly equal for the most commonly studied graph types. Appendix~\ref{section::exact_compute} further discusses how the upper bound for $\rho$ in Theorem~\ref{theorem::boundsupper} can be improved computationally following the criterion of Lemma~\ref{lemma::general}.

Given that the Gaussian ranks are not readily accessible, graph-based bounds are required to inform statistical practice using the prescribed pseudo-likelihood (and Gaussian) methods. Combining our result from Theorems~\ref{theorem::boundsupper} and \ref{theorem::boundslower} with known lower bounds for $\gamma$, the following can be established for the pseudo-likelihood ranks:

\begin{corollary} \label{theorem::bounds}
    For any graph $\mathcal{G}$, $\kappa^*+1 \leq \rho \leq \min\set{\rho^*, \ell} \leq \max\set{\rho^*, \ell}\leq \delta+1$.
\end{corollary}

\begin{corollary} \label{theorem::bounds3}
    For any graph $\mathcal{G}$, $\kappa^*+1 \leq \gamma \leq \rho \leq \ell \leq \delta+1$.
\end{corollary}

The bounds given in Corollaries~\ref{theorem::bounds} and \ref{theorem::bounds3} for the pseudo-likelihood ranks are the same as those of \cite{BenDavid2015, gross_maximum_2018} for the Gaussian \textsc{mle}. The upper bound informs the sample size required for practical implementation of both methods, so in this sense a pseudo-likelihood approach is competitive with its Guassian counterpart. Furthermore, for the most common classes of graphs studied in the graphical models literature, the outer bounds are known to coincide (i.e. $\kappa^* + 1 = \delta + 1$), giving exact correspondence $\gamma=\gamma^*=\rho=\rho^*$. This includes: trees, homogeneous graphs, circular graphs, rectangular grids, complete graphs, chordal graphs, and complete bipartite graphs; precise values may be found in Table~\ref{table::common}.

\begin{remark}
    Other bounds have been previously proposed for the Gaussian rank. \citep{10.2307/4616281} gives an upper bound $\gamma \leq \MIN\set{d, \operatorname{tw}} + 1$ in terms of the maximum degree $d$ and treewidth $\operatorname{tw}$. Natural questions are: (i) Is this bound applicable in the pseudo-likelihood context? (ii) How does it compare with graph degeneracy? $\delta \leq d$ elementarily: by definition there exists a subgraph of $\mathcal{G}$ whose minimum degree is $\delta$. Since $d$ instead equals the maximum degree over all subgraphs, $\delta \leq d$. Likewise, the treewidth equals the minimum degeneracy over all chordal covers of $\mathcal{G}$. Degeneracy increases with subgraph inclusion, implying $\delta \leq \operatorname{tw}$. Theorem~\ref{theorem::boundsupper} gives bounds for the pseudo-likelihood ranks in terms of graph degeneracy $\rho^* \leq \delta+1$, thus $\rho^* \leq \MIN\set{d, \operatorname{tw}}+1$ is also an applicable, but looser, upper bound.
\end{remark}

Although we have considered a deterministic regime for the purposes of our work, where the minimizer is guaranteed to exist, the results of Section~\ref{section::general} apply more broadly. From Lemmas~\ref{lemma::general} and \ref{lemma::general_2}, the \textsc{concord} and \textsc{conspace} pseudo-likelihood estimators exist or fail to exist together, and their (joint) existence implies that of the Gaussian \textsc{mle}. As such, we expect that the Gaussian \textsc{mle} will exist and be unique with uniformly higher probability for any sample size or data generating mechanism, even outside the deterministic regime delineated by the pseudo-likelihood and Gaussian ranks. To better understand the properties of our estimators in this ``probabilistic regime'', Appendix~\ref{section::numerical} provides a numerical investigation, and compares the relative probability of existence for both the pseudo-likelihood and Gaussian problems.

\pagebreak

\appendix

\section{Proofs of main results} \label{section::proofs}

\setcounter{lemmasupp}{0}
\begin{lemmasupp}
	Given $A \SUCCEQ 0$, consider either pseudo-likelihood objective $\con_{A}$ or $\spl_{A}$. Then a corresponding minimizer exists on $\DGPP$ if and only if
	\begin{equation}
		\KDGP{A} = \KDGZ{A}. \nonumber
	\end{equation}
	
	Furthermore, if a minimizer $\Omega^* \IN \DGPP$ exists, the set of minimizers is affine of the form
	\begin{equation}
		\set{\Omega^*+\Phi\colon \Phi \IN \KDGZ{A}}. \nonumber
	\end{equation}
\end{lemmasupp}

\begin{proof}
	Consider first the functions $g_1\colon\reals^2\to\reals$ and $g_2\colon\reals\times\p{0,\infty}\to\reals$ given by $g_1\p{u,v; c}\coloneqq\p{u+cv}^2$ and $g_2\p{u,v; c}\coloneqq\p{u+cv}^2/v$. Both are convex in $\p{u,v}$ for any $c\in \reals$. Furthermore, $g_1\p{u_0+\alpha u^{*}, v_0+\alpha v^{*}; c}$ is linear in $\alpha$ when $u^*+cv^*=0$ and strictly convex otherwise. Likewise $g_2\p{u_0+\alpha u^{*}, v_0+\alpha v^{*}; c}$ is linear in $\alpha$ when $u^*v_0-v^*u_0=0$ and strictly convex otherwise.
	
	Given $A \SUCCEQ 0$, we have $A=\Gamma^2$ for some unique matrix $\Gamma \SUCCEQ 0$. For the objective $\con_{A}$, we may write $\con_{A}\p{\Omega}=-2\sum_i\LOG\p{\Omega_{ii}}+\sum_i\sum_k\p{\Gamma_{ki}\Omega_{ii}+\sum_{j\neq i}\Gamma_{kj}\Omega_{ji}}^2=-2\sum_i\LOG\p{\Omega_{ii}}+\sum_i\sum_k g_1\p{\sum_{j\neq i}\Gamma_{kj}\Omega_{ji}, \Omega_{ii}; \Gamma_{ki}}$. Note each term is convex in $\Omega$. Supposing a minimizer $\Omega^* \in \DGPP$ exists, then to construct the minimizer set we need only find the directions $\Phi\in \SYMG$ in which $\con_{A}\p{\Omega^*+\alpha\Phi}$ (and by extension each of its terms) is linear (therefore constant) in $\alpha$. $-\LOG\p{\cdot}$ is strictly convex, so $\Phi_{ii}=0$ necessarily for all $i$. Thus $\sum_{j\neq i}\Gamma_{kj}\Phi_{ji}=0$ for any $i,k$ from the linearity criterion of $g_1$. Summarizing, we have $\Phi_{ii}=0$ and $\Gamma\Phi_{\cdot i}=0$ for all $i$. Then $\Phi\in \KDGZ{A}$. Because $\con_{A}\p{\Omega^*+\alpha\Phi}$ is constant in $\alpha$ for any $\Phi \in \KDGZ{A}$, we have fully characterized the minimizer set.
	
	Without loss of generality, we may remove the constant direction $\KDGZ{A}$ from consideration. Let $T$ be the orthogonal complement of $\KDGZ{A}$ within $\SYMG$. We may restrict the problem to $T$ implicitly by projection along $\KDGZ{A}$, on which the objective is constant. Include now the additional assumption $\KDGZ{A} = \set{0}$ after restriction to $T$.
	
	$\p{\implies}$ Now assume there exists an element $\Phi \in \KDGP{A} \setminus \set{0}$. Then $\Phi_{ii}>0$ for some $i$ because $\KDGZ{A} = \set{0}$. For any $\alpha>0$ and $\varepsilon>0$ there exists an element $\Psi \in \DGPP$ such that $\norm[\mathrm{F}]{\Psi-\alpha\Phi} \leq \varepsilon$. We have $\Psi_{ii}\geq \alpha\Phi_{ii}-\varepsilon$ and $\TR\p{\Psi^{\mathsmaller{\top}}A\Psi} \leq \p{\sqrt{\TR\p{\Phi^{\mathsmaller{\top}}A\Phi}}+\sqrt{\TR\p{\p{\Psi-\alpha\Phi}^{\mathsmaller{\top}}A\p{\Psi-\alpha\Phi}}}}^2 = \TR\p{\p{\Psi-\alpha\Phi}^{\mathsmaller{\top}}A\p{\Psi-\alpha\Phi}} \leq \norm[\mathrm{F}]{\Psi-\alpha\Phi}^2 \leq \lambda_{\mathrm{max}}\varepsilon^2$, for some $\lambda_{\mathrm{max}}>0$ depending only on $A$. Letting $\alpha\to\infty$, $\epsilon\to 0$ gives $\con_{A}\p{\Psi} \to -\infty$, and thus no minimizer exists.
	
	$\p{\impliedby}$ When $\KDGP{A} = \set{0}$, then because $\DGP$ is a closed cone and $K_{A}$ is a linear space there exists a $\lambda_{\mathrm{min}}>0$ such that $ \TR\p{\Omega^{\mathsmaller{\top}}A\Omega}\geq\lambda_{\mathrm{min}}\norm[\mathrm{F}]{\Omega}^2$ on $\DGP$. We have $\con_{A}\p{\Omega}\leq -\p{p/2}\LOG\p{\norm[\mathrm{F}]{\Omega}}+\lambda_{\mathrm{min}}\norm[\mathrm{F}]{\Omega}^2$, which goes to $\infty$ uniformly as $\norm[\mathrm{F}]{\Omega}\to\infty$. Thus we reduce the minimization to a closed subset of $\DGP$, excluding the boundaries where the objective becomes uniformly large. Because the objective is continuous, a minimizer is guaranteed to exist.
	
	For the objective $\spl_{A}$ a similar argument applies. Again, suppose a minimizer $\Omega^* \in \DGPP$ exists. We may write $\spl\p{\Omega}=-2\sum_i\LOG\p{\Omega_{ii}}+\mybrk\sum_i\sum_k g_2\p{\sum_{j\neq i}\Gamma_{kj}\Omega_{ji}, \Omega_{ii}; \Gamma_{ki}}$. For any direction $\Phi\in \SYMG$ such that $\spl\p{\Omega^*+\alpha\Phi}$ is linear in $\alpha$, we have $\Phi_{ii}=0$ for all $i$. Therefore because $\p{\Omega^*}_{ii}>0$, from the directional linearity criterion of $g_2$ we have $\sum_{j\neq i}\Gamma_{kj}\Phi_{ji}=0$ for all $k$. Then $\Phi\in \KDGZ{A}$. As with the previous case, $\spl_{A}\p{\Omega^*+\alpha\Phi}$ is constant in $\alpha$ for any $\Phi \in \KDGZ{A}$, thus we have fully characterized the minimizer set for $\spl_{A}$ as well. Without loss of generality we may remove the constant direction, assuming further that $\KDGZ{A} = \set{0}$.
	
	Now to characterize existence of a minimizer for the \textsc{conspace} objective, first reparametrize $\spl_{A}$. The map $\Omega_{ji}\to\Omega_{ji}/\sqrt{\Omega_{ii}}\eqqcolon \beta_{ji}$ is invertible and continuous. The reparametrized objective becomes $-4\sum_i\LOG\p{\beta_{ii}}+\sum_i\beta_{\cdot i}^{\mathsmaller{\top}}A\beta_{\cdot i}$ constrained such that $\beta_{ji}\beta_{ii}=\beta_{ij}\beta_{jj}$ for all $i$ and $j$ and $\beta_{ii}>0$ for all $i$. Likewise, given a graph $\mathcal{G}$, the corresponding model constraints give $\beta_{ij}=\beta_{ji}=0$ if $\set{i,j}\in E\p{\mathcal{G}}$. Note that the reparametrized objective contains the same quadratic form as $\con_{A}$, with kernel $K_{A}$. Denote the reparametrized constraint space as $\RGPP$, which is a cone in $\mathcal{M}$. The closure is $\RGP$, in which the diagonal positivity constraint is relaxed to include $\beta$ such that $\beta_{ii}\geq 0$ for all $i$, while still satisfying the other two constraints. Note that $\RGP$ also depends nonlinearly on $T$, but the image under reparametrization is, importantly, still a closed cone. By the same arguments previously made, a minimizer fails to exist whenever $\KRGP{A} \setminus \set{0} \neq \emptyset$, and exists when $\KRGP{A} = \set{0}$. Now remove the reparametrization. $\Omega_{\cdot i}/\sqrt{\Omega_{ii}}=0$ if and only if $\Omega_{\cdot i}=0$, thus $\KRGP{A} \setminus \set{0} \neq \emptyset$ if and only if $\KDGP{A}\setminus\set{0}\neq \emptyset$. The result follows.
\end{proof}

\setcounter{lemmasupp}{1}
\begin{lemmasupp}
	Given $A \SUCCEQ 0$, a unique minimizer of $G_{A}$ exists on $\PGPP$ if and only if $$\KPGP{A}{\mskip 1mu}{=}{\mskip 1mu}\set{0}.$$
\end{lemmasupp}

\begin{proof}
	
	Because $A \SUCCEQ 0$, we may write $A=\Gamma^2$ for some unique $\Gamma \SUCCEQ 0$. Then for $\Omega \SUCCEQ 0$, $ \TR\p{A\Omega}=\sum_i\Gamma_{\cdot i}^{\mathsmaller{\top}}\Omega\Gamma_{\cdot i}\geq 0$. Because $\Omega \SUCCEQ 0$, $\Gamma_{\cdot i}^{\mathsmaller{\top}}\Omega\Gamma_{\cdot i}=0 \iff \Omega\Gamma_{\cdot i}=0$. Then $A\Omega=\sum_i\Gamma_{\cdot i}\Gamma_{\cdot i}^{\mathsmaller{\top}}\Omega=0$. We have $\set{\Omega\in \mathcal{M}\colon \TR\p{A\Omega}\leq 0} \cap \PGP = \KPGP{A}$.
	
	$\p{\implies}$ The proof is similar to that of Lemma~\ref{lemma::general}, and for brevity we omit overlapping arguments. Assume first that there exists an element $\Phi \in \KPGP{A} \setminus \set{0}$. Then for any $\Omega^*\in\PGPP$ and $\alpha>0$ define $\Psi\coloneqq \Omega^*+\alpha\Phi$. We have $\TR\p{A\Psi}=\TR\p{A\Omega^*}$ and $\LOG\abs{\Psi} \geq \LOG\p{\p{\TR\p{\Omega^*}+\alpha\TR\p{\Phi}}\lambda_\mathrm{min}\p{\Omega^*}^{p-1}/p}=\infty$, where $\lambda_{\mathrm{min}}\p{\Omega^*}>0$ is the smallest eigenvalue of $\Omega^*$. Letting $\alpha\to\infty$ gives $\gaus_{A}\p{\Psi}\to-\infty$, and we conclude no minimizer exists on $\PGPP$.
	
	$\p{\impliedby}$ Now assume $\KPGP{A} = \set{0}$. Then $\set{\Omega\in \mathcal{M}\colon \TR\p{A\Omega}\leq 0} \cap \PGP = \set{0}$. Because $\PGP$ is a pointed closed cone, there exists some $\lambda_1>0$ such that $ \TR\p{A\Omega}\geq \lambda_1\norm[\mathrm{F}]{\Omega}$. Then $\gaus_{A}\p{\Omega}\geq -\p{p/2}\LOG\p{\norm[\mathrm{F}]{\Omega}}+\lambda_1\norm[\mathrm{F}]{\Omega}$. This bound goes to $\infty$ uniformly as $\norm[\mathrm{F}]{\Omega}\to\infty$, and the objective approaches $\infty$ near the other boundaries of $\PGPP$, thus we conclude the unique minimizer exists.
\end{proof}

\setcounter{lemmasupp}{2}
\begin{lemmasupp}
	$\rho_{{\mskip 2mu}\textsc{concord}} = \rho_{{\mskip 2mu}\textsc{conspace}} \leq \rho_{{\mskip 2mu}\textsc{conspace}}^{*} = \rho_{{\mskip 2mu}\textsc{concord}}^{*}$.
\end{lemmasupp}

\begin{proof}
    Equality between the \textsc{concord} and \textsc{conspace} ranks follows immediately from Lemma~\ref{lemma::general}, showing that the existence properties of the \textsc{conspace} and \textsc{concord} estimators are completely equivalent. Furthermore, almost every matrix in $\SYM$ is in general position, thus $\rho_{\textsc{concord}} \leq \rho^*_{\textsc{concord}}$.
\end{proof}


\setcounter{corollarysupp}{0}
\begin{corollarysupp}
    For any graphs $\mathcal{G}_1 \subset \mathcal{G}_2$, $\rho\p{\mathcal{G}_1} \leq \rho\p{\mathcal{G}_2}$ and $\rho^*\p{\mathcal{G}_1} \leq \rho^*\p{\mathcal{G}_2}$.
\end{corollarysupp}

\begin{proof}
    $\SYMG[1]\subset\SYMG[2]$ are the linear constraint sets associated with the given graphs. For $K_{A}=\set{\Omega\in \mathcal{M}\colon A\Omega=0}$, we have that $\KDGZ[2]{A}=\set{0}\implies \KDGZ[1]{A}=\KDGZ[2]{A}{\cap}{\mskip 3mu}\SYMG[1]=\set{0}$. Likewise, $\KDGP[2]{A} = \KDGZ[2]{A} \implies \KDGP[1]{A} = \KDGP[2]{A}{\cap}{\mskip 3mu}\SYMG[1] = \KDGZ[2]{A}{\cap}{\mskip 3mu}\SYMG[1] = \KDGZ[1]{A}$. Via Lemma~\ref{lemma::general}, we see that existence of the pseudo-likelihood minimizers given $\mathcal{G}_2$ implies the same given $\mathcal{G}_1$, likewise for uniqueness. The result follows.
\end{proof}

\setcounter{theoremsupp}{0}
\begin{theoremsupp}[Upper bound]
    For any graph $\mathcal{G}$, both $\rho \leq \ell \leq \delta+1$ and $\rho^* \leq \delta+1$.
\end{theoremsupp}

\setcounter{theoremsupp}{1}
\begin{theoremsupp}[Lower bound]
    For any graph $\mathcal{G}$, both $\gamma \leq \rho $ and $\gamma^* \leq \rho^* $.
\end{theoremsupp}

\setcounter{corollarysupp}{1}
\begin{corollarysupp}
    For any graph $\mathcal{G}$, $\kappa^*+1 \leq \rho \leq \min\set{\rho^*, \ell} \leq \max\set{\rho^*, \ell}\leq \delta+1$.
\end{corollarysupp}

\setcounter{corollarysupp}{2}
\begin{corollarysupp}
    For any graph $\mathcal{G}$, $\kappa^*+1 \leq \rho \leq \min\set{\rho^*, \ell} \leq \max\set{\rho^*, \ell}\leq \delta+1$.
\end{corollarysupp}

\begin{proof}
    We will establish the above four results together, starting with Theorem~\ref{theorem::boundsupper}. Let a graph $\mathcal{G}$ be given where node $i$ has degree $d_i$. First we will establish the upper bound for $\rho^*=\rho^*\p{\mathcal{G}}$. Assume that $\RANK\p{A}\geq \delta\p{\mathcal{G}}+1$, where $\delta=\delta\p{\mathcal{G}}$ is the degeneracy.
    
    Consider the quadratic form $\sum_i \omega_i^{\mathsmaller{\top}}A \omega_i$. Each term is nonnegative, and the kernel is $K_{A}\coloneqq\set{\Omega\in \mathcal{M}\colon A\Omega=0}$. We may rewrite this form as $\sum_i \omega_i^{\mathsmaller{\top}}A_i \omega_i$, where only the nonzero elements defined by the graphical model are included, and $A_i$ is the submatrix of $A$ associated with the incident edges of $i$. Each $A_i$ has rank at least $\MIN\set{d_i+1, \delta+1}$ when $A$ is in general position. Via the definition of graph degeneracy, there is at least one node $j$ with degree equal to $\delta$, thus $\omega_j^{\mathsmaller{\top}}A_j \omega_j$ is strictly convex. To construct an element of the kernel, $\omega_{jk}=0$ for any edge $\set{j,k}$ incident with $j$. This reduction is equivalent to removing node $j$ from the graph. The problem reduces. Because each subgraph of $\mathcal{G}$ has at least one node with degree at least $\delta$, we may eventually eliminate all vertices in this fashion, and thus conclude strict convexity for the initial quadratic form. We have $\KSYMG{A}=\set{0}$. By Lemma~\ref{lemma::general} a unique minimizer for $\con_{A}$ exists. We conclude $\rho^*\leq \delta+1$.

    The space of positive semi-definite matrices of rank $r$ is parametrized by the map $X\to X^{\mathsmaller{\top}}X$, where $X$ is a $r\times p$ real matrix. Since this map is polynomial in the elements of $X$, and $X^{\mathsmaller{\top}}X\Omega=0$ if and only if $X\Omega=0$, then $\KSYMG{A}=\set{0}$ for almost every $A \SUCCEQ 0$ if and only if $\KSYMG{X}=\set{0}$ for almost every $r\times p$ matrix $X$. By Theorem~6.3 in \citep{gross_maximum_2018}, this property holds for any $r\geq\ell$. Since $\KSYMG{A}=\set{0}\implies \KDGP{A}=\set{0}$, we conclude $\rho\leq \ell$ via the criterion of Lemma~\ref{lemma::general}. Likewise, since the general position matrices are a specific dense set, the initial argument made above can be used to show that $\ell\leq \delta+1$.
    
    Next we compare the Gaussian and pseudo-likelihood ranks. For a minimizer of the \textsc{concord} and \textsc{conspace} objectives to exist and be unique, via Lemma~\ref{lemma::general} equivalently $\KDGP{A} = \set{0}$. $\PP \subset \DP \implies \KPGP{A} = \set{0}$, thus via Lemma~\ref{lemma::general_2} the unique Gaussian minimizer exists. Theorem~\ref{theorem::boundslower} follows.

    Corollaries~\ref{theorem::bounds} and \ref{theorem::bounds3} are established via the derived inequality $\gamma \leq \rho \leq \rho^* \leq \delta+1$. Subsequently $\kappa^*+1 \leq \gamma$ via \citep[Theorem~4.2]{BenDavid2015}, where $\kappa^* = \kappa^*\p{\mathcal{G}}$ is the subgraph connectivity.
\end{proof}

\section{Exact values of the Gaussian and pseudo-likelihood ranks}

\begin{table}[h]
 \begin{center}
 \begin{tabular}{|l|c|c|}
 \hline
 Graph type & $\gamma = \rho$ \\
 \hline
 $\operatorname{Tree}$ & 2 \\
 $\operatorname{Circular}\p{p}$& 3 \\
 $\operatorname{Rectangular Grid}$& 3 \\
 $\operatorname{Complete}\p{p}$ & $p$ \\
 $\operatorname{Homogeneous}$ & $\delta$ \\
 $\operatorname{Chordal}$ & $\delta$ \\
 $\operatorname{Complete Bipartite}\p{m,n}$ & $\MIN\set{m,n}$ \\
 \hline
 \end{tabular}
 \end{center}
 
 \caption{Exact values of the Gaussian rank $\gamma$ and the weak pseudo-likelihood rank $\rho$ for various graph types. Graph degeneracy is denoted by $\delta$. Note that the Gaussian and pseudo-likelihood ranks are the same for all classes of graph in the table.}
 \label{table::common}
\end{table}

\section{Directly computing upper bounds for the weak pseudo-likelihood rank} \label{section::exact_compute}

To compute $\rho^*\p{\mathcal{G}}$, from Lemma~\ref{lemma::general} we must characterize $\KDGP{A} = \KSYMG{A}\cap\set{\Omega\in{\mathcal M}\colon \Omega_{ii}\geq 0 \ \forall i =1,\ELPS, p}$. A minimizer of $\con_{A}$ or $\spl_{A}$ exists on $\DGPP$ and is unique if and only if $\KDGP{A}=\KDGZ{A}=\set{0}$. Let $\pi_D\colon\mathcal{M}\to\reals^p$ denote the projection operator such that $\pi_D\p{\Omega}=\pi_D\p{\Omega_D}=\p{\Omega_{ii}}_{i=1}^p$. Then the criterion $\KDGP{A}=\set{0}$ can be reduced to the equivalent $\pi_D\p{\KSYMG{A}}\cap [0, \infty)^p=\set{0}$ and $\KDGZ{A}=0$.

Given $\mathcal{G}$, the set of constraint equations defining $\SYMG$ are $\Omega_{ij}-\Omega_{ji}=0$ for all $\set{i,j}\in\mathcal{G}$ and $\Omega_{ij}=\Omega_{ji}=0$ for all $\set{i,j}\notin G$. Then there exists a matrix $C_{\mathcal G}$ encoding these linear equations such that $\Omega\in\SYMG\iff C_{\mathcal G}\operatorname{vec}\p{\Omega}=0$.

Furthermore, consider the set of positive semi-definite matrices of rank $r$. The set is parametrized under the map $X\to X^{\mathsmaller{\top}}X$, where $X$ is any $r\times p$ real matrix. We have $X^{\mathsmaller{\top}}X\Omega=0\iff X\Omega=0$. Furthermore, because the map is polynomial in the elements of $X$, the pre-image of any negligible set is negligible. Therefore $\KDGP{A}=\set{0}$ for almost every $A\succeq 0$ with $\rank\p{A}=r$ if and only if $\KDGP{X}=\set{0}$ for almost every $r\times p$ real matrix. Proceeding, there exists a matrix $B_r$ encoding the linear equations $X\Omega=0$ such that $\Omega\in K_X\iff B_r\operatorname{vec}\p{\Omega}=0$. The elements of $B_r$ are polynomial in the elements of $X$. The kernel of $D=[C_{\mathcal{G}}^{\mathsmaller{\top}}, B_r^{\mathsmaller{\top}}]^{\mathsmaller{\top}}$ encodes the space $\KSYMG{X}$. Let the columns of $D$ be ordered so that columns associated with off-diagonal elements come before those associated with diagonals, and further assume $D$ to be in row-echelon form. The nullity of $D$ excluding the $p$ final columns gives $\dim\p{\KDGZ{X}}$. Increment $r$ until $\dim\p{\KDGZ{X}}=0$; this can be accomplished without recalculation by successively introducing new blocks to $B_r$ and reducing. We have $\rho^*\p{\mathcal{G}}\geq r$. Then let $D'$ be the submatrix created by removing columns and pivotal rows associated with the off-diagonal elements; $\ker\p{D'}$ corresponds to $\pi_D\p{\KSYMG{X}}$. The problem then reduces to ensuring $\ker\p{D'}\cap [0,\infty)^p=\set{0}$ almost everywhere. $D'$ is a $\rank\p{D'}\times p$ polynomial matrix in the elements of $X$. We increment $r$ until the following checks are satisfied. Let $\bar{r}\p{\mathcal G}$ be the smallest such $r\in\nats$:

Consider the set $\ker\p{D'}\cap [0, \infty)^p$. Assume that $\ker\p{D'}$ is not entirely contained in any of the $p$ linear boundary spaces of $[0, \infty)^p$, otherwise the problem reduces. For the condition $\ker\p{D'}\cap [0, \infty)^p=\set{0}$ to hold, either $\nullity\p{D'}=1$ or $\ker\p{D'}\cap [0, \infty)^{i-1}\cap\set{0}\cap[0, \infty)^{p-i}=\set{0}$ for all $i=1,\ELPS,p$. In the second case, the criterion can be checked recursively on each of the $p$ boundary faces, avoiding repeat checks where possible. The intersection of $\ker\p{D'}$ and the $i$\textsuperscript{th} boundary face is given by the kernel of $D'$ with the $i$\textsuperscript{th} column removed, and so no recalculations involving $D'$ are required at this step. In the first case, if $\ker\p{D'}$ is not contained in any boundary face, $\ker\p{D'}\cap [0, \infty)^p=\set{0}\iff \ker\p{D'}^{\mathsmaller{\bot}}\cap [0, \infty)^p\neq \set{0}$. $\ker\p{D'}^{\mathsmaller{\bot}}$ corresponds to the row space $\operatorname{rowsp}\p{D'}$.

To ensure $\operatorname{rowsp}\p{D'}\cap [0, \infty)^p\neq \set{0}$, we apply a similar recursive approach. As before, either $\rank\p{D'}=1$ or $\operatorname{rowsp}\p{D'}\cap [0, \infty)^{i-1}\cap\set{0}\cap[0, \infty)^{p-i}\neq \set{0}$ for at least one of $i=1,\ELPS,p$. In the second case, we proceed recursively to each face, again avoiding rechecks where possible, continuing until some nonzero intersection is found. In the case of $\rank\p{D'}=1$, assume $\operatorname{rowsp}\p{D'}$ is not entirely contained in any of the boundary faces. Then $\operatorname{rowsp}\p{D'}\cap [0, \infty)^p\neq \set{0}$ for almost every $X$ only if $D'\geq 0$ elementwise. The set of $X$ for which any nontrivial components of $D'$ are $0$ constitutes a non-trivial variety, and thus has measure $0$. For our purposes, this set can be ignored.

Note: recursion in our described fashion above continues until arriving at the $1$-dimensional base cases, of which many are possible. The nonnegativity condition holding for one of these given a specific $X$ is enough to ensure that the initialized $\operatorname{rowsp}\p{D'}\cap [0, \infty)^p\neq \set{0}$. However, it need not be the same base case satisfying the condition for all $X$. Thus, requiring that one base case satisfy nonnegativity for all $X$ is a sufficient condition, but not necessary to ensure that $\operatorname{rowsp}\p{D'}\cap [0, \infty)^p\neq \set{0}$ for almost every $X$. Checks of this kind only give an upper bound for $\rho^*\p{\mathcal{G}}$, unless equalling the initializing lower-bound given previously. The computational complexity of the non-negativity checks can be improved by using a stronger condition, for example checking whether the polynomial elements can be written as sums-of-squares. However, the resultant upper-bound will be larger as a result, thus giving an accuracy-complexity tradeoff.

Note that when $\KSYMG{X}=\set{0}$, as for ranks $r$ greater than or equal to the generic completion rank $\ell\p{G}$, the starting matrix $D'$ is full rank, in which case all later checks succeed, regardless of the specific non-negativity criterion used. With this in mind, the computational upper bound $\bar{r}\p{\mathcal{G}}\leq \ell\p{\mathcal G}$ improves upon this existing result.

\section{Computing the generic completion rank} \label{section::algorithm}

To compute $\ell\p{\mathcal{G}}$, from Lemma~\ref{lemma::general} we must characterize $\KDGP{A}$. For $\RANK\p{A} \eqqcolon r$, assume $\set{\gamma_{1}, \ELPS, \gamma_{r}}$ is such that $Ax=0\iff \gamma_i^Tx=0\ \forall i=1,\ELPS,r$. Let $A_{ij}$ denote the matrix whose $i$-th row equals $\gamma_j$, $0$ otherwise. We have $\mybrk \bigcap_{i=1}^{p}\bigcap_{j=1}^{r}\set{\Omega\in\mathcal{M}\colon\TR\p{A_{ij}\Omega}=0}=K_{A}$. This translates to a vectorized system of linear equations. Furthermore, let $B_{ij}$ denote the matrix such that $\p*{B_{ij}}_{ij} = \smash{-\p{B_{ij}}}_{ji} = 1$, $0$ otherwise. Finally, let $C_{ij}$ be the matrix such that $\p{C_{ij}}_{ij}=0$, $0$ otherwise. We have $\p{\bigcap_{\set{i,j}\in\mathcal{G}}\set{\Omega\in\mathcal M\colon \TR\p{B_{ij}\Omega}=0}}\cap\mybrk\p{\bigcap_{\set{i,j}\notin\mathcal{G}}\set{\Omega\in\mathcal M\colon \TR\p{C_{ij}\Omega}=0}}=\SYMG$. The intersection $\KSYMG{A}$ can then be written as the joint intersection $\p{\bigcap_{i=1}^{p}\bigcap_{j=1}^{r}\set{\Omega\in\mathcal{M}\colon\TR\p{A_{ij}\Omega}=0}}\cap\mybrk\p{\bigcap_{\set{i,j}\in\mathcal{G}}\set{\Omega\in\mathcal M\colon \TR\p{B_{ij}\Omega}=0}}\cap\p{\bigcap_{\set{i,j}\notin\mathcal{G}}\set{\Omega\in\mathcal M\colon \TR\p{C_{ij}\Omega}=0}}$. This set can be represented as the kernel of a large sparse matrix $\brk{\mathcal{A}_1^{\mathsmaller{\top}}, \ELPS, \mathcal{A}_r^{\mathsmaller{\top}}, \mathcal {B}_{\mathcal{G}}^{\mathsmaller{\top}}, \mathcal{C}_{\mathcal{G}}^{\mathsmaller{\top}}}^{\mathsmaller{\top}}$, with each block associated with the relevant solution set above. Then to find the smallest value of $r$ such that $\DIM\p{\KSYMG{A}}=0$, we need only successively add blocks $\mathcal{A}_i$ until the corresponding matrix is full-rank. To implement this, standard algorithms apply, e.g. Gaussian elimination. Let ``\textsc{Reduce}'' denote a chosen basis reduction algorithm. Note that equations in $\mathcal{C}_{\mathcal{G}}$ stemming from the $C_{ij}$ can be essentially ignored for implementation, as can any reduced equations without non-pivotal elements. The corresponding pivotal columns can simply be ignored without further computation, and the target rank changed to account.

Generally, $\DIM\p{\KSYMG{A}}$ is not constant, instead varying with $\gamma_{1}, \ELPS, \gamma_{r}$. However, this dimension is constant almost everywhere. The sets above can interpreted as solving a set of polynomial-valued equations in the components of $\gamma_{1}, \ELPS, \gamma_{p-r}$. Under this polynomial interpretation, $R^*\coloneqq\DIM\p{K_{A}\cap\SYMG}$ is a deterministic function of $r$ and $\mathcal{G}$. The set of specific $\gamma_{1}, \ELPS, \gamma_{p-r}$ for which $\DIM\p{\KSYMG{A}}>R^*$ constitutes a nontrivial algebraic variety, and thus has measure zero.

As seen in the proof of Theorem~\ref{theorem::boundsupper}, the generic completion rank $\ell$ equals the smallest value $r = \RANK\p{A}$ such that $R^*=p^2$ (alternatively $R^*=p+2\cdot\#\set{E\p{\mathcal G}}$ after eliminating $\mathcal{C}_{\mathcal{G}}$). Implemented using polynomial arithmetic with rational coefficients, $\ell$ is exactly computable in this fashion. A sparse implementation is recommended due to the large volume of zeros in the equation sets. A formal description of this computational algorithm is provided in Algorithm~\ref{algo::upper}.\\

\begin{algorithm}[H]
    \DontPrintSemicolon
    \KwIn{A graph $\mathcal{G}$}
    \KwOut{$\ell\p{\mathcal{G}}$}
    $\mathcal{C} \gets \mathcal{B}_{\mathcal{G}}$\;
    \For{$r \gets 1$ \textbf{to} p-1}{
        $\mathcal{C} \gets$ {\sc Reduce}${\mskip 1mu}\p{\brk{\mathcal{C}^{\mathsmaller{\top}}, \mathcal{A}_r^{\mathsmaller{\top}}}^{\mathsmaller{\top}}}$\;
        \If{$\rank\p{\mathcal{C}}=p+2\cdot\#\set{E\p{\mathcal{G}}}$}{
            \Return{r}\;
        }
    }
    \Return{p};
    \caption{the generic completion rank of a general graph}
    \label{algo::upper}
\end{algorithm}

\section{Numerical results} \label{section::numerical}

To supplement the theoretical results established in Section~\ref{section::ranks}, a simulation study was performed to investigate the relative existence properties of the \textsc{concord}/\textsc{conspace} and Gaussian \textsc{ml} estimators outside of the deterministic regime (i.e. the ``probabilistic regime''). Shown in Figures~\ref{fig::prob1} and \ref{fig::prob2} are the probabilities of existence and uniqueness for both methods over several example graphs, based on $\mathcal{N}\p{0, I_p}$ data. Relative probabilities are consistent with our general theoretical results from Section~\ref{section::general}, with uniformly lower probabilities expected for the pseudo-likelihood approach. This follows from Lemmas~\ref{lemma::general} and \ref{lemma::general_2} which show that the existence of a unique minimizer for the pseudo-likelihood objectives implies the existence of the Gaussian \textsc{mle}.

\begin{figure}
    \begin{center}
    \includegraphics[width=\textwidth]{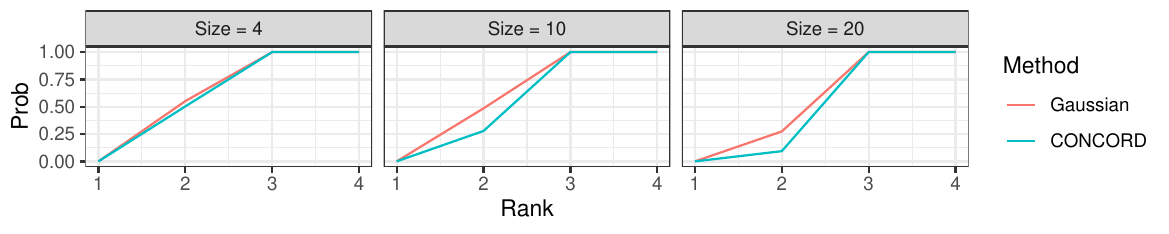}
    \end{center}
    \caption{Probability of existence of the Gaussian and \textsc{concord}-\textsc{conspace} estimators when the underlying graph is a cycle. ``Size'' indicates the number of variables, or equivalently the number of vertices in the graph.}
    \label{fig::prob1}
\end{figure}

\begin{figure}
    \begin{center}
    \includegraphics[width=\textwidth]{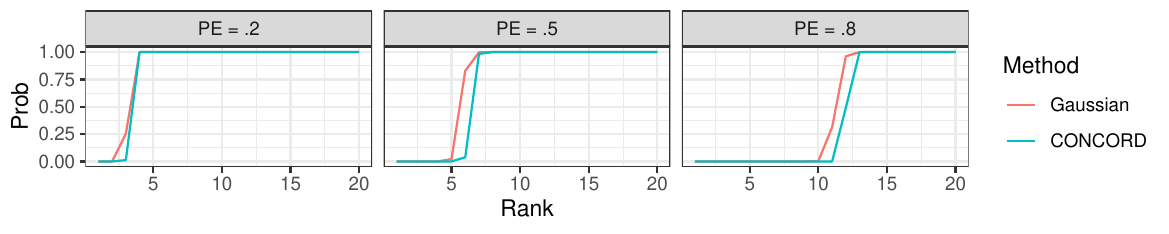}
    \end{center}
    \caption{Probability of existence of the Gaussian and \textsc{concord}/\textsc{conspace} estimators when the underlying graphs are Erd\H{o}s-R\'{e}nyi of size $20$ with edge probability $\text{PE}$, conditioned so that each graph contains a single connected component. Left: $\text{PE}=0.3$, Middle: $\text{PE}=0.5$, Right: $\text{PE}=0.8$.}
    \label{fig::prob2}
\end{figure}

\end{document}